\documentclass[12pt]{article}
\usepackage{amssymb}
\usepackage{mathrsfs}
\usepackage[hypertex]{hyperref}
\usepackage{amsfonts}
\usepackage{graphicx}
\usepackage{latexsym,amsmath,amssymb,amsfonts,amsthm}

\newcommand{\bcen}{\begin{center}}
\newcommand{\ecen}{\end{center}}
\newtheorem{theorem}{Theorem}[section]
\newtheorem{lemma}[theorem]{Lemma}

\newtheorem{corollary}[theorem]{Corollary}
\newtheorem{proposition}[theorem]{Proposition}
\newtheorem{remark}[theorem]{Remark}

\newtheorem{defn}[theorem]{Definition}

\begin{document}
\setcounter{page}{1}
\title{On eigenfunctions and nodal sets of the Witten-Laplacian}
\author{Ruifeng Chen,~~ Jing Mao$^{\ast}$,~~Chuanxi Wu}

\date{}
\protect \footnotetext{\!\!\!\!\!\!\!\!\!\!\!\!{~$\ast$ Corresponding author\\
MSC 2020:
35P15, 49Jxx, 35J15.}\\
{Key Words: Witten-Laplacian, Neumann eigenvalues, Laplacian, the
free membrane problem, isoperimetric inequalities. } }
\maketitle ~~~\\[-15mm]

\begin{center}
{\footnotesize Faculty of Mathematics and Statistics,\\
 Key Laboratory of Applied
Mathematics of Hubei Province, \\
Hubei University, Wuhan 430062, China\\
 Key Laboratory of Intelligent Sensing System and Security (Hubei
University), Ministry of Education\\
Emails: gchenruifeng@163.com (R. F. Chen), jiner120@163.com (J.
Mao), \\ cxw@hubu.edu.cn (C. X. Wu)
 }
\end{center}


\begin{abstract}

In this paper, we successfully establish a Courant-type nodal domain
theorem for both the Dirichlet eigenvalue problem and the closed
eigenvalue problem of the Witten-Laplacian. Moreover, we also
characterize the properties of the nodal lines of the eigenfunctions
of the Witten-Laplacian on smooth Riemannian $2$-manifolds. Besides,
for a Riemann surface with genus $g$, an upper bound for the
multiplicity of closed eigenvalues of the Witten-Laplacian can be
provided.
 \end{abstract}


\section{Introduction} \label{sect-1}
\renewcommand{\thesection}{\arabic{section}}
\renewcommand{\theequation}{\thesection.\arabic{equation}}
\setcounter{equation}{0}

Let $(M^{n},\langle\cdot,\cdot\rangle)$ be an $n$-dimensional
($n\geq2$) complete Riemannian manifold with the metric
$\langle\cdot,\cdot\rangle$. Let $\Omega\subseteq M^{n}$ be a
 domain in $M^n$, and $\phi\in C^{\infty}(M^n)$ be a
smooth\footnote{~In fact, one might see that $\phi\in C^{2}$ is
suitable to derive our main conclusions in this paper. However, in
order to avoid putting too much attention on discussion for the
regularity of $\phi$ and follow the assumption on conformal factor
$e^{-\phi}$ for the notion of \emph{smooth metric measure spaces} in
many literatures (including of course those cited in this paper),
without specification, we prefer to assume that $\phi$ is smooth on
the domain $\Omega$. } real-valued function defined on $M^n$. In
this setting, on $\Omega$, the following elliptic operator
\begin{eqnarray*}
\Delta_{\phi}:=\Delta-\langle\nabla\phi,\nabla\cdot\rangle
\end{eqnarray*}
can be well-defined, where $\nabla$, $\Delta$ are the gradient and
the Laplace operators on $M^{n}$, respectively. The operator
$\Delta_{\phi}$ w.r.t. the metric $\langle\cdot,\cdot\rangle$ is
called the \emph{Witten-Laplacian} (also called the \emph{drifting
Laplacian} or the \emph{weighted Laplacian}). The $K$-dimensional
Bakry-\'{E}mery Ricci curvature $\mathrm{Ric}^{K}_{\phi}$ on $M^{n}$
can be defined as follows
 \begin{eqnarray*}
\mathrm{Ric}^{K}_{\phi}:=\mathrm{Ric}+\mathrm{Hess}\phi-\frac{d\phi\otimes
d\phi}{K-n},
 \end{eqnarray*}
where $\mathrm{Ric}$ denotes the Ricci curvature tensor on $M^{n}$,
and $\mathrm{Hess}$ is the Hessian operator on $M^{n}$ associated to
the metric $\langle\cdot,\cdot\rangle$. Here $K>n$ or $K=n$ if
$\phi$ is a constant function. When $K=\infty$, the so-called
$\infty$-dimensional Bakry-\'{E}mery Ricci curvature
$\mathrm{Ric}_{\phi}$ (simply, \emph{Bakry-\'{E}mery Ricci
curvature} or \emph{weighted Ricci curvature}) can be defined as
follows
\begin{eqnarray*}
 \mathrm{Ric}_{\phi}:=\mathrm{Ric}+\mathrm{Hess}\phi.
\end{eqnarray*}
These notions were introduced by D. Bakry and M. \'{E}mery in
\cite{BE}. Many interesting results (under suitable assumptions on
the Bakry-\'{E}mery Ricci curvature) have been obtained -- see,
e.g., \cite{DMWW,XDL,JM1,WW}. Besides, we refer readers to
\cite[Section 1]{CM1} for a nice introduction about the motivation,
the meaning, and the reason why people considered geometric problems
under the assumption of Bakry-\'{E}mery Ricci curvature bounded.

For a bounded domain $\Omega$ (with boundary $\partial\Omega$) in a
given complete $n$-manifold $M^{n}$, $n\geq2$, consider the
Dirichlet eigenvalue problem of the Witten-Laplacian as follows
\begin{eqnarray} \label{eigen-D11}
\left\{
\begin{array}{ll}
\Delta_{\phi}u+\lambda u=0\qquad & \mathrm{in}~\Omega\subset M^{n}, \\[0.5mm]
u=0 \qquad & \mathrm{on}~{\partial\Omega},
\end{array}
\right.
\end{eqnarray}
and it is not hard to check that the operator $\Delta_{\phi}$ in
(\ref{eigen-D11}) is \textbf{self-adjoint} w.r.t. the following
inner product
\begin{eqnarray}  \label{inn-p}
\widehat{(h_{1},h_{2})}:=\int_{\Omega}h_{1}h_{2}d\eta=\int_{\Omega}h_{1}h_{2}e^{-\phi}dv,
\end{eqnarray}
 $h_{1},h_{2}\in\widehat{W}^{1,2}_{0}(\Omega)$,  where
$d\eta:=e^{-\phi}dv$ with $dv$ the Riemannian volume element of
$M^{n}$, and $\widehat{W}^{1,2}_{0}(\Omega)$ stands for a Sobolev
space, which is the completion of the set of smooth functions (with
compact support) $C^{\infty}_{0}(\Omega)$ under the following
Sobolev norm
\begin{eqnarray}  \label{Sobo-1}
\widehat{\|f\|}_{1,2}:=\left(\int_{\Omega}f^{2}d\eta+\int_{\Omega}|\nabla
f|^{2}d\eta\right)^{1/2}.
\end{eqnarray}
Clearly, if $\phi=const.$ is a constant function, then
$\Delta_{\phi}$ degenerates into the Laplacian $\Delta$, and
(\ref{eigen-D11}) becomes the classical fixed membrane problem (i.e.
the Dirichlet eigenvalue problem of the Laplacian). Since
$\Delta_{\phi}$ in (\ref{eigen-D11}) is self-adjoint w.r.t. the
inner product $\widehat{(\cdot,\cdot)}$, it is natural to ask:
 \begin{itemize}
\item \emph{Whether or not $\Delta_{\phi}$ in (\ref{eigen-D11}) has some similar
spectral properties as the Dirichlet Laplacian defined in
$\Omega\subset M^{n}$?}
 \end{itemize}
 The answer is affirmative. For instance, as we have pointed out in
 \cite[Section 1]{CM1} or \cite[pp. 789-790]{DM1}, using
 similar arguments\footnote{~In fact, as explained in \cite{CM1,DM1}, if one wants to get similar spectral
properties for the Dirichlet Witten-Laplacian, the key point is
using the weighted Riemannian volume element $d\eta=e^{-\phi}dv$ to
replace the usual one $dv$ in proofs of those spectral properties of
the eigenvalue problem of the Dirichlet Laplacian. Of course, one
needs to check carefully whether or not every step still works for
the weighted case. } to those of the classical fixed membrane
problem of the Laplacian (i.e., the discussions about the existence
of discrete spectrum, Rayleigh's theorem, Max-min theorem, Domain
monotonicity of eigenvalues with vanishing Dirichlet data, etc.
Those discussions are standard, and for details, please see for
instance \cite{IC}), it is not hard to know:
\begin{itemize}
\item The self-adjoint elliptic operator $-\Delta_{\phi}$ in
(\ref{eigen-D11}) \emph{only} has discrete spectrum, and all the
 eigenvalues in this discrete spectrum can be listed
non-decreasingly as follows
\begin{eqnarray} \label{sequence-3}
0<\lambda_{1,\phi}(\Omega)<\lambda_{2,\phi}(\Omega)\leq\lambda_{3,\phi}(\Omega)\leq\cdots\uparrow+\infty.
\end{eqnarray}
Each eigenvalue $\lambda_{i,\phi}$, $i=1,2,\cdots$, in the sequence
(\ref{sequence-3}) was repeated according to its multiplicity (which
is finite and equals to the dimension of the eigenspace of
$\lambda_{i,\phi}$). By applying the standard variational
principles, one can obtain that the $k$-th Dirichlet eigenvalue
$\lambda_{k,\phi}(\Omega)$ can be characterized as
follows\footnote{~For simplification and without specification, in
the sequel, we prefer to write $\lambda_{i,\phi}(\Omega)$ as
$\lambda_{i,\phi}$. Similar thing will be done for the closed
eigenvalues  $\lambda_{i,\phi}^{c}(\Omega)$ too.}
 \begin{eqnarray}  \label{chr-1}
 \lambda_{k,\phi}(\Omega)=\inf\left\{\frac{\int_{\Omega}|\nabla f|^{2}e^{-\phi}dv}{\int_{\Omega}f^{2}e^{-\phi}dv}
 \Bigg{|}f\in\widehat{W}^{1,2}_{0}(\Omega),f\neq0,\int_{\Omega}ff_{i}e^{-\phi}dv=0\right\},
 \end{eqnarray}
where $f_{i}$, $i=1,2,\cdots,k-1$, denotes an eigenfunction of
$\lambda_{i,\phi}(\Omega)$. Moreover, the first Dirichlet eigenvalue
$\lambda_{1,\phi}(\Omega)$ of the eigenvalue problem
(\ref{eigen-D11}) satisfies
\begin{eqnarray}  \label{chr-2}
 \lambda_{1,\phi}(\Omega)=\inf\left\{\frac{\int_{\Omega}|\nabla f|^{2}d\eta}{\int_{\Omega}f^{2}d\eta}
 \Bigg{|}f\in\widehat{W}^{1,2}_{0}(\Omega),f\neq0\right\}.
 \end{eqnarray}
\end{itemize}
If the bounded domain $\Omega\subset M^{n}$ does not have boundary,
then the closed eigenvalue problem of the Witten-Laplacian
\begin{eqnarray} \label{eigen-C11}
\Delta_{\phi}u+\lambda u=0 \qquad \mathrm{in}~\Omega
\end{eqnarray}
can be considered. Similar to the Dirichlet case, it is not hard to
know:
\begin{itemize}
\item The self-adjoint elliptic operator $-\Delta_{\phi}$ in
(\ref{eigen-C11}) \emph{only} has discrete spectrum, and all the
 eigenvalues in this discrete spectrum can be listed
non-decreasingly as follows
\begin{eqnarray} \label{sequence-3-1}
0=\lambda^{c}_{0,\phi}<\lambda^{c}_{1,\phi}(\Omega)\leq\lambda^{c}_{2,\phi}(\Omega)\leq\lambda^{c}_{3,\phi}(\Omega)\leq\cdots\uparrow+\infty.
\end{eqnarray}
It is easy to see that the first closed eigenvalue
$\lambda^{c}_{0,\phi}=0$ is simple and its eigenfunctions are
nonzero constant functions.
 Each nonzero eigenvalue $\lambda^{c}_{i,\phi}$, $i=1,2,\cdots$, in
the sequence (\ref{sequence-3-1}) was repeated according to its
multiplicity (which is finite and equals to the dimension of the
eigenspace of $\lambda^{c}_{i,\phi}$). Applying the standard
variational principles, one can obtain that the $k$-th nonzero
closed eigenvalue $\lambda^{c}_{k,\phi}(\Omega)$ can be
characterized as follows
 \begin{eqnarray}  \label{chr-1-1}
 \lambda^{c}_{k,\phi}(\Omega)=\inf\left\{\frac{\int_{\Omega}|\nabla f|^{2}e^{-\phi}dv}{\int_{\Omega}f^{2}e^{-\phi}dv}
 \Bigg{|}f\in\widehat{W}^{1,2}(\Omega),f\neq0,\int_{\Omega}ff_{i}e^{-\phi}dv=0\right\},
 \end{eqnarray}
where $f_{i}$, $i=1,2,\cdots,k-1$, denotes an eigenfunction of
$\lambda^{c}_{i,\phi}(\Omega)$. Moreover, the first nonzero closed
eigenvalue $\lambda^{c}_{1,\phi}(\Omega)$ of the eigenvalue problem
(\ref{eigen-C11}) satisfies
\begin{eqnarray}  \label{chr-2-1}
 \lambda^{c}_{1,\phi}(\Omega)=\inf\left\{\frac{\int_{\Omega}|\nabla f|^{2}d\eta}{\int_{\Omega}f^{2}d\eta}
 \Bigg{|}f\in\widehat{W}^{1,2}(\Omega),f\neq0,\int_{\Omega}fe^{-\phi}dv=0\right\}.
 \end{eqnarray}
 Here $\widehat{W}^{1,2}(\Omega)$ is the admissible space of the
 eigenvalue problem (\ref{eigen-C11}), and is actually the completion of the set of smooth functions  $C^{\infty}(\Omega)$ under the
Sobolev norm (\ref{Sobo-1}).
\end{itemize}
Of course, for a bounded domain with boundary, similar to the
Laplacian case, one can also propose the Neumann eigenvalue problem
and the mixed eigenvalue problem for the Witten-Laplacian. One can
see, e.g. \cite{CM,CM1}, for some interesting spectral properties of
the Neumann eigenvalue problem of the Witten-Laplacian.

Maybe one might find previous contents (i.e. the existence
conclusion on discrete spectrum and also the characterizations of
eigenvalues of the eigenvalue problems (\ref{eigen-D11}) and
(\ref{eigen-C11}) of the Witten-Laplacian) in other literatures (not
ours), but for the completion of the structure of this paper and
also for the convenience of readers, here we prefer to write them
down.

In order to state our main conclusions clearly, we need the
following concept, which can be found, e.g., in \cite{IC}.

\begin{defn}
Let $f:M^{n}\rightarrow\mathbb{R}\in C^{0}$. Then the \textbf{nodal
set} is the set $f^{-1}[0]$, and a \textbf{nodal domain} of $f$ is a
component on $\overline{M^n}\setminus f^{-1}[0]$.
\end{defn}

For the eigenvalue problems (\ref{eigen-D11}) and (\ref{eigen-C11}),
first we can get the following Courant-type nodal domain theorem.

\begin{theorem} \label{theo-1}
(Nodal domain theorem for the Witten-Laplacian) Assume that $\Omega$
is a regular domain on a given Riemannian $n$-manifold $M^{n}$. For
the Dirichlet eigenvalue problem (\ref{eigen-D11}), its eigenvalues
consist of a non-decreasing sequence (\ref{sequence-3}). Denote by
$f_{i}$ an eigenfunction of the $i$-th eigenvalue
$\lambda_{i,\phi}$, $i=1,2,3,\cdots$, and
$\{f_{1},f_{2},f_{3},\cdots\}$ forms a complete orthogonal basis of
${\widehat{L}^2(\Omega)}$. Then for each $k=1,2,3,\cdots$, the
number of nodal domains of $f_{k}$ is less than or equal to $k$.
Similar conclusion can be obtained for the closed eigenvalue problem
(\ref{eigen-C11}) as well, that is, for each $k=0,1,2,3,\cdots$, the
number of nodal domains of an eigenfunction of the $k$-th closed
eigenvalue $\lambda^{c}_{k,\phi}$ is less than or equal to $k+1$.
\end{theorem}

\begin{remark}
\rm{ (1) As the Laplacian case, one could also try to get
Courant-type nodal domain results for the Neumann eigenvalue problem
and the mixed eigenvalue problem of the Witten-Laplacian (by using a
similar argument to the one given in Section \ref{sect-2}). However,
in this paper, we just focus on the Dirichlet eigenvalues and closed
eigenvalues of the Witten-Laplacian. \\
(2) Maybe spectral geometers have already known the conclusion of
Theorem \ref{theo-1}, but for the completion of the structure of
this paper, we still write it down formally. Another reason is that
in our another recent work \cite{CM}, we have used the conclusion of
Theorem \ref{theo-1} to get the Hong-Krahn-Szeg\H{o} type
isoperimetric inequalities for the second Dirichlet eigenvalue of
the Witten-Laplacian on bounded domains in both the Euclidean spaces
and the hyperbolic spaces. BTW, by making necessary changes to the
proof of Courant-type theorem for the characterization of nodal
domains to eigenfunctions of the Laplacian in the Riemannian case
given by B\'{e}rard-Meyer \cite{BM}, one might get our proof of
Theorem \ref{theo-1} shown in Section \ref{sect-2}.}
\end{remark}

Second, we can get several interesting properties concerning the
nodal set and the multiplicity of eigenvalues.

\begin{theorem}  \label{theo-2}
Assume that $M^2$ is a smooth Riemannian $2$-manifold  without
boundary (not necessarily compact). If a smooth function $f\in
C^\infty(M^2)$ satisfies $(\Delta_\phi+h(x))f=0$, $h\in
C^\infty(M^2)$, then except for a few isolated points, the set of
nodes of $f$ forms a smooth curve.
\end{theorem}

\begin{corollary} \label{coro-1}
Assume that $M^2$ is a Riemannian $2$-manifold. Then, for any
solution of the equation $(\Delta_\phi+h(x))f=0$, $h\in
C^\infty(M)$, one has:
 \begin{itemize}
\item The critical points on the nodal lines are isolated.

\item When the nodal lines meet, they form an equiangular system.

\item The nodal lines consist of a number of $C^2$-immersed closed
smooth curves. Therefore, when $M^2$ is compact, they are a number
of $C^2$-immersed circles. A $C^2$-immersed circle means
$\Phi(\mathbb{S}^1)$, where $\Phi:\mathbb{S}^1\rightarrow M^2$ is a
$C^2$ immersion.
\end{itemize}
\end{corollary}

\begin{theorem} \label{theo-3}
Suppose that $M^2$ is a compact Riemann surface of genus $g$. Then
the multiplicity of the $i$-th closed eigevnalue
$\lambda^{c}_{i,\phi}(M^2)$ is less than or equal to
$(2g+i+1)(2g+i+2)/2$.
\end{theorem}

This paper is organized as follows. The proof of Theorem
\ref{theo-1} will be given in Section \ref{sect-2}. The proofs of
Theorems \ref{theo-2} and \ref{theo-3}, Corollary \ref{coro-1} will
be shown in Section \ref{sect-3}.

\section{Proof of Theorem \ref{theo-1}} \label{sect-2}
\renewcommand{\thesection}{\arabic{section}}
\renewcommand{\theequation}{\thesection.\arabic{equation}}
\setcounter{equation}{0}

We divide the proof of Theorem \ref{theo-1} into two cases:
\begin{itemize}
\item \textbf{Case I}. When $f_{k}$'s nodal domains are
normal domains;

\item \textbf{Case II}. When there is no assumption made on the
nodal domains, we only discuss the closed and Dirichlet eigenvalue
problems of the Witten-Laplacian, i.e. the eigenvalue problems
(\ref{eigen-D11}) and (\ref{eigen-C11}).
\end{itemize}
For \textbf{Case I}, let $G_1, G_2,\cdots, G_k, G_{k+1}, \cdots$ be
the nodal domains of $f_{k}$, and define
\begin{eqnarray*}
\varphi_{j}:=f_{k}|_{G_j}, \qquad \mathrm{on}~ G_j
\end{eqnarray*}
and
\begin{eqnarray*}
\varphi_{j}:=0, \qquad \mathrm{on}~ \overline{\Omega}\setminus G_j
\end{eqnarray*}
for each $j=1,2,\cdots,k$. It is not hard to know that there exists
a nontrivial function
\begin{eqnarray} \label{2-1}
\mathcal{F}:=\sum\limits_{j=1}^{k}\alpha_{j}\varphi_{j}
\end{eqnarray}
satisfying
\begin{eqnarray*}
\widehat{(\mathcal{F},f_{1})}=\widehat{(\mathcal{F},f_{2})}=\cdots=\widehat{(\mathcal{F},f_{k-1})}=0,
\end{eqnarray*}
which is equivalent with
 \begin{eqnarray} \label{2-2}
\sum\limits_{j=1}^{k}\alpha_{j}\widehat{(\varphi_{j},f_{l})}=0,
\qquad l=1,2,\cdots,k-1.
 \end{eqnarray}
 In fact, if one treats $\alpha_{1},\cdots,\alpha_{k}$ as
unknowns and $\widehat{(\varphi_{j},f_{l})}$ as coefficients, then
the system (\ref{2-2}) has more unknowns than equations, which
implies a nontrivial solution
$(\alpha_{1},\alpha_{2},\cdots,\alpha_{k})$ must exist. Then using a
very similar argument to the proof of Rayleigh's theorem of the
Laplacian (see e.g. \cite[Chapter I]{IC}), one can obtain
 \begin{eqnarray} \label{2-3}
\lambda_{k,\phi}(\Omega)\leq\frac{\int_\Omega|\nabla\mathcal{F}|^2d\eta}{\int_{\Omega}
\mathcal{F}^2d\eta}.
\end{eqnarray}
On the other hand, by the definition of $\mathcal{F}$ (i.e.
(\ref{2-1})), it is not hard to see
\begin{eqnarray} \label{2-4}
\frac{\int_\Omega|\nabla\mathcal{F}|^2d\eta}{\int_\Omega\mathcal{F}^2d\eta}&=&\frac{\sum\limits_{j,l=1}^{k}\int_\Omega\alpha_{j}\nabla\varphi_{j}\alpha_{l}\nabla\varphi_{l}
d\eta}{\sum\limits_{j,l=1}^{k}\int_\Omega \alpha_{j}\varphi_{j}\alpha_{l}\varphi_{l}d\eta} \nonumber \\
&=&
\frac{\sum\limits_{j,l=1}^{k}\left(\sum\limits_{i}\int_{G_i}\alpha_{j}\nabla\varphi_{j}\alpha_{l}\nabla\varphi_{l}
d\eta\right)}{\sum\limits_{j,l=1}^{k}\left(\sum\limits_{i}\int_{G_i}
\alpha_{j}\varphi_{j}\alpha_{l}\varphi_{l}d\eta\right)}\nonumber\\
 &=& \frac{\sum\limits_{i=1}^{k}\int_{G_i}\alpha^{2}_{i}|\nabla\varphi_{i}|^{2}
d\eta}{\sum\limits_{i=1}^{k}\int_{G_i}
\alpha^{2}_{i}\varphi^{2}_{i}d\eta}.
\end{eqnarray}
Since for each $i=1,2,\cdots,k$, by the divergence theorem and the
definition of $\varphi_{i}$, it follows that
\begin{eqnarray*}
-\int_{G_i}\varphi_{i}\Delta_{\phi}\varphi_{i}d\eta=\int_{G_i}|\nabla\varphi_{i}|^2d\eta=\lambda_{k,\phi}
\int_{G_i}\varphi_{i}^{2}d\eta,
\end{eqnarray*}
combining this fact with (\ref{2-4}) yields
\begin{eqnarray} \label{2-5}
\frac{\int_\Omega|\nabla\mathcal{F}|^2d\eta}{\int_\Omega\mathcal{F}^2d\eta}=\lambda_{k,\phi}.
\end{eqnarray}
Therefore, from (\ref{2-3}) and (\ref{2-5}), one easily knows that
$\mathcal{F}$ is an eigenfunction corresponding to
$\lambda_{k,\phi}$, and $\mathcal{F}|_{G_{k+1}}=0$. Then using the
maximum principle of the second-order elliptic PDEs to the system
(\ref{eigen-D11}), it follows that  $\mathcal{F}\equiv0$ in
$\Omega$. This is a contradiction. So, this completes the proof of
the assertion of Theorem \ref{theo-1} for Dirichlet eigenvalues of
the Witten-Laplacian in \textbf{Case I}.

Making suitable minor changes to the above argument, one can get the
nodal domain result for closed eigenvalues of the Witten-Laplacian
directly.

In \textbf{Case II}, the divergence theorem cannot be directly used
for calculations. Therefore, we first need to establish an integral
formula for the nodal domains of eigenfunctions.

\begin{lemma} \label{lemma2-1}
Assume that $u\in C^2$, $v\in C^1$, and $D$ is a nodal domain of the
eigenfunction $u$. Then we have
\begin{eqnarray*}
-\int_Dv \Delta_{\phi}u d\eta=\int_D \nabla u\cdot\nabla vd\eta.
\end{eqnarray*}
\end{lemma}

\begin{proof}
We may assume that $u$ is positive in $D$. For each $\epsilon>0$,
define $D_{\epsilon}:=\{x\in D|u(x)>\epsilon\}$ and a function
$u_{\epsilon}$ as follows:
\begin{eqnarray} \label{2-6}
u_{\epsilon}=\left\{
\begin{array}{ll}
u-\epsilon,~~&\mathrm{in}~D_{\epsilon}, \\[1mm]
0,~~&\mathrm{otherwise}.
\end{array}
\right.
 \end{eqnarray}
Assume that $\{\epsilon_i\}$ is a non-increasing sequence of regular
values of $u$, and $\epsilon_i$ approaches $0$ as $i$ tends to
infinity (BTW, the existence of this sequence can be assured by
Sard's theorem). Let $u_{i}=u_{\epsilon_i}$, $D_i=D_{\epsilon_i}$,
and $\partial D_i$ be the smooth boundary of $D_i$. Applying the
Green's formula to $D_i$, we have
 \begin{eqnarray} \label{2-7}
-\int_{D_i} v \Delta_{\phi} u_i d\eta=\int_{D_i} \nabla u_i \cdot
\nabla v d\eta.
 \end{eqnarray}
Using (\ref{2-6}) and (\ref{2-7}), one can further obtain
\begin{eqnarray} \label{2-8}
\left|-\int_{D} v \Delta_{\phi} u d\eta-\int_{D}\nabla v \cdot\nabla ud\eta\right|&=&\Bigg{|}\int_{D} v \Delta_{\phi} u d\eta - \int_{D_i} v \Delta_{\phi} u_i d\eta- \nonumber\\
 &&\qquad \int_{D_i} \nabla u_i \cdot \nabla v d\eta +  \int_{D}\nabla v \cdot\nabla ud\eta\Bigg{|} \nonumber \\
&\leq& \left|\int_{D\setminus D_i} v \Delta_{\phi} u d\eta\right|  + \left|\int_{D\backslash D_i}\nabla v  \cdot\nabla u d\eta\right|\nonumber\\
&\leq& c|D\setminus D_i|_{\phi},
\end{eqnarray}
where $c$ is a positive constant, and $|\cdot|_{\phi}$ stands for
the weighted volume of a prescribed geometric object. Since
$D=\bigcup_{i=1}^\infty D_i$, it is easy to know that $|D\setminus
D_i|_{\phi}$ tends to $0$ as $i$ tends to infinity. This fact,
together with (\ref{2-8}), implies the conclusion of Lemma
\ref{lemma2-1} directly.
\end{proof}

\begin{lemma} \label{lemma2-2}
Assume that $u$ is an eigenfunction of the eigenvalue problem
(\ref{eigen-D11}). Let $\lambda$ be the eigenvalue corresponding to
$u$, and $D$ be a nodal domain of $u$. Then we have
$\lambda=\lambda_{1,\phi}(D)$, where, by the convention of the usage
of notations in Section \ref{sect-1}, $\lambda_{1,\phi}(D)$
naturally denotes the first Dirichlet eigenvalue of the
Witten-Laplacian on $D$.
\end{lemma}

\begin{proof}
It is not hard to see that the eigenfunction $u$ is at least $C^2$.
By applying Lemma \ref{lemma2-1}, one has
 \begin{eqnarray*}
\lambda\int_D u^2 d\eta=-\int_D u\Delta_{\phi} u d\eta=\int_D
|\nabla u|^2 d\eta \geq \int_{D_\epsilon} |\nabla u|^2d\eta \geq
\lambda_{1,\phi}(D_\epsilon)\int_{D_\epsilon} u_{\epsilon}^2d\eta,
 \end{eqnarray*}
with the domain $D_{\epsilon}$ and the function $u_{\epsilon}$
constructed as in the proof of Lemma \ref{lemma2-1}. Due to the fact
$\overline{D_\epsilon}\subset D$, and using the Domain monotonicity
of Dirichlet eigenvalues of the Witten-Laplacian (see e.g.
\cite[Lemma 1.5]{DM1}), one knows
$\widehat{W}^{1,2}_{0}(D_\epsilon)\subset\widehat{W}^{1,2}_{0}(D)$
and $\lambda_{1,\phi}(D_\epsilon)\geq \lambda_{1,\phi}(D)$.
Therefore, it follows that
 \begin{eqnarray*}
\lambda\int_D u^2 d\eta \geq \lambda_{1,\phi}(D)\int_{D_\epsilon}
u_{\epsilon}^2 d\eta.
 \end{eqnarray*}
Then, letting $u_\epsilon$ approach $u$ in the sense of
$\widehat{L}^2$, one has
 \begin{eqnarray} \label{2-9}
\lambda \geq \lambda_{1,\phi}(D).
 \end{eqnarray}

On the other hand, define a positive function $v_{\epsilon}\in
\widehat{W}^{1,2}_{0}(D_\epsilon)$ which satisfies the
 equation $\Delta_{\phi}
v_\epsilon=-\lambda_{1,\phi}(D_\epsilon)v_\epsilon$. Using the fact
that $u$ is an eigenfunction of $\lambda$ directly yields
\begin{eqnarray} \label{2-10}
-\int_{D_\epsilon} v_\epsilon\Delta_{\phi}ud\eta=\lambda
\int_{D_\epsilon}v_\epsilon u d\eta.
\end{eqnarray}
Since the boundary $\partial D_{\epsilon}$ is smooth, applying the
Green's formula, it is not hard to get
 \begin{eqnarray} \label{2-11}
-\int_{D_\epsilon} v_\epsilon\Delta_{\phi} ud\eta&=&-\int_{D_\epsilon}u \Delta_{\phi} v_\epsilon d\eta+\int_{\partial D_\epsilon}u \frac{\partial v_\epsilon}{\partial \vec{\nu}}d\eta\nonumber\\
&=&\lambda_{1,\phi}(D_\epsilon)\int_{D_\epsilon}u v_\epsilon d\eta
+\int_{\partial D_\epsilon}u \frac{\partial v_\epsilon}{\partial
\vec{\nu}}d\eta,
 \end{eqnarray}
where $\vec{\nu}$ denotes the outward unit normal vector field along
$\partial D_{\epsilon}$. Since $\frac{\partial v_\epsilon}{\partial
\vec{\nu}}\leq 0$, we have from (\ref{2-10})-(\ref{2-11}) that
$\lambda_{1,\phi}(D_\epsilon)\geq \lambda$. For any positive number
$a>0$, by using the characterization (\ref{chr-2}) with $\Omega=D$,
there should exist $\rho\in C^{\infty}(D)$, which is compactly
supported on $D$, such that
$\lambda_{1,\phi}(D)\geq\inf\left\{\frac{\int_D |\nabla \rho|^2
d\eta}{\int_D \rho^2 d\eta}\right\}-a$. Based on the choice of the
function $\rho$, there must exist $\epsilon>0$ for which
$\mathrm{supp}(\rho)\subseteq D_{\epsilon}$, and then by using the
characterization (\ref{chr-2}) with $\Omega=D_\epsilon$, one gets
$\lambda_{1,\phi}(D_\epsilon)\leq \inf\left\{\frac{\int_D |\nabla
\rho|^2 d\eta}{\int_D \rho^2 d\eta}\right\}$ directly. So, for a
given $a>0$, there exists $\epsilon>0$ such that
\begin{eqnarray} \label{2-12}
\lambda_{1,\phi}(D)\geq \lambda_{1,\phi}(D_\epsilon)-a\geq\lambda-a.
\end{eqnarray}
Combining (\ref{2-9}) and (\ref{2-12}), using the fact that
$\lambda_{1,\phi}(D_{\epsilon})$ is increasing with respect to
$\epsilon$, and the arbitrariness of $a>0$, one can easily obtain
$\lambda_{1,\phi}(D)=\lambda$.
\end{proof}

As one knows that if $D$ is a bounded domain, $u\in
\widehat{W}^{1,2}(D)\cap C^0(\overline{D})$ and $u|_{\partial D}=0$,
then $u\in\widehat{W}_{0}^{1,2}(D)$. So, using Lemma \ref{lemma2-1},
Lemma \ref{lemma2-2}, and a similar argument to the proof of
\textbf{ Case I}, we can get the nodal domain theorem for the
Witten-Laplacian in \textbf{Case II}.

\section{Nodal lines and multiplicities of eigenvalues of the
Witten-Laplacian} \label{sect-3}
\renewcommand{\thesection}{\arabic{section}}
\renewcommand{\theequation}{\thesection.\arabic{equation}}
\setcounter{equation}{0}

Assume that $f$ satisfies the equation $(\Delta_{\phi}+h(x))f=0$,
$h\in C^{\infty}(M^n)$, and assume that $x_0$ is a zero point of
$f(x)$, i.e. $f(x_0)=0$. So, one can assume that $M^n$ locates
within a small neighborhood of $x_0$. We use normal coordinates
around $x_0$ and hence we can assume we are working in a small open
set of the origin of $\mathbb{R}^n$. Using the results of N.
Aronszajn \cite{NA}, $f$ can only vanish up to finite order around
the origin. Hence, by L. Bers' main result in \cite{LB} (see also
\cite[Theorem 2.1]{SYC}), it follows that
\begin{eqnarray}
f(x)=\mathcal{P}_{N}(x)+O(|x|^{N+\varepsilon}),
\end{eqnarray}
where $\mathcal{P}_N$ is a homogenous polynomial of degree $N$ and
$\varepsilon\in(0,1)$. Besides, $\mathcal{P}_N$ also satisfies the
osculating equation at the origin. Since the normal coordinates
around $x_0$ have been used, $\mathcal{P}_N$ should be a spherical
harmonic polynomial of degree $N$. In other words, $\mathcal{P}_N$
satisfies
 \begin{eqnarray*}
\left(\sum_{i=1}^{n} \frac{\partial^2}{\partial
x_i^2}\right)\mathcal{P}_N(x)=0.
 \end{eqnarray*}
If $N=1$, $\mathcal{P}_{N}(x)$ is a linear polynomial and
$df(0)\neq0$, the set of nodes around $0$ is a smooth curve. Below,
we will discuss the case that $N$ is strictly greater than $1$.

First, we need the following result.
\begin{lemma} \label{lemma2-3}
(\cite[Lemma 2.3]{SYC})  Assume that $\mathcal{P}_N$ is a
homogeneous harmonic polynomial of degree $N$, $N>1$. Then the set
of nodes of $\mathcal{P}_{N}(x)$ around the origin  has a
singularity at $0$.
\end{lemma}

We also need:
\begin{lemma} \label{lemma2-4}
(\cite[Lemma 2.4]{SYC}) Suppose that $f$, $p$ are smooth function in
$\mathbb{R}^n$,
 \begin{eqnarray*}
&&f(x)=p(x)+O(|x|^{N+\varepsilon})\\
&&\frac{\partial f(x)}{\partial x_i}=\frac{\partial p(x)}{\partial x_i}+O(|x|^{N-1+\varepsilon}),\quad N\geq1, \quad \varepsilon\in(0,1)\\
&&\frac{\partial^v }{\partial x_1^{i_1}...x_n^{i_n}}p(0)=0,\quad
0\leq v\leq N-1
\end{eqnarray*}
and
\begin{eqnarray*}
|\nabla p|\geq const|x|^{N-1}.
\end{eqnarray*}
Then there exists a local $C^1$ diffeomorphism $\Phi$ fixing the
origin such that
\begin{eqnarray*}
f(x)=p(\Phi (x)).
\end{eqnarray*}
\end{lemma}

\begin{proof} [Proof of Theorem \ref{theo-2}]
The nodal set at the origin for $\mathcal{P}_N$ is equal to
$\{tx:~t>0,\mathcal{P}_N|_{\mathbb{S}^1}(x)=0\}$. Assume that $\pi$
is a set consisting of single points, and then
$\mathcal{P}^{-1}_N(0)\setminus \pi=\gamma_0$ is a smooth curve.
Therefore, by Lemma \ref{lemma2-4}, we can obtain that
$f^{-1}(0)\setminus \Phi^{-1}(\pi)=\Phi^{-1}(\gamma_0)$ and
$\Phi^{-1}(\gamma_0)$ is $C^1$ curve. Assume
$y\in\Phi^{-1}(\gamma_0)$, and then we have $f(y)=0$ and
$\Phi(y)\in\gamma_0$. So, by L. Bers' main result in \cite{LB} (see
also \cite[Theorem 2.1]{SYC}), we have $f(x)\sim
\mathcal{P}_{\hat{N}}(x)$ around the point $y$ for some $\hat{N}>0$.
Assume that $\hat{N}>1$. Because $\Phi^{-1}(\gamma_0)$ is a $C^1$
diffeomorphism in a small neighborhood near $y$ and the nodal set of
$\mathcal{P}_{\hat{N}}$ near the origin is $C^1$ diffeomorphic, but
Lemma \ref{lemma2-3} shows that if $\hat{N}>1$, then
$\mathcal{P}_{\hat{N}}$ should have singularity at $0$, which
contradicts with the fact that $\Phi^{-1}(\gamma_0)$ is a $C^1$
smooth curve. This means our assumption cannot be true, and so
$\hat{N}=1$. Summing up, $\Phi^{-1}(\gamma_0)$ can only be a smooth
curve, which completes the proof of Theorem \ref{theo-2}.
\end{proof}

The assertions of Corollary \ref{coro-1} can be obtained by directly
applying Theorem \ref{theo-2}.

At the end, we would like to give a characterization to
multiplicities to the closed eigenvalues of the Witten-Laplacian on
compact Riemann surface with genus $g$. However, first we need the
following proposition, which is a direct consequence of Theorem
\ref{theo-1}.

\begin{proposition} \label{prop-1}
Assume that $M^2$ is a compact $2$-dimensional smooth manifold with
or without boundary. We have:

(1) If the boundary $\partial M^2\neq\emptyset$ is non-empty, then
the number of nodal domains of the $i$-th Dirichlet eigenvalue
$\lambda_{i,\phi}$ is less than or equal to $i$;

(2) If $\partial M^2 = \emptyset$, then the number of nodal domains
of the $i$-th closed eigenvalue $\lambda^{c}_{i,\phi}$ is less than
or equal to $i+1$.
\end{proposition}

We also need:

\begin{lemma} \label{lemma3-4}
(\cite[Lemma 3.1]{SYC}) Assume that $M^2$ is a compact Riemann
surface with genus $g$, and $\psi_i:\mathbb{S}^1\rightarrow M^2$ is
a closed $C^1$ curve on $M^2$. If
$\psi_i(\mathbb{S}^1)\cap\psi_k(\mathbb{S}^1)$, $i\neq k$, consists
of only a finite number of points, then
$M^{2}\setminus\psi_1(\mathbb{S}^1)\cup\cdots\cup\psi_{2g+t}(\mathbb{S}^1)$
is at least $(t+1)$-connected.
\end{lemma}

Assume that $u$ is an eigenfunction of $\lambda^{c}_{i,\phi}$. If
$x_0$ is a zero point of $u$ (i.e. $u(x_0)=0$), and the zero order
of $x_0$ is $k$, then by Corollary \ref{coro-1}, we know that there
will be $k$ nodal lines starting from $x_0$. So, if $k>2g+i$, then
by Lemma \ref{lemma3-4}, one has that $M^2$ is at least
$(i+1)$-connected. However, by Proposition \ref{prop-1}, $M^2$
should be at most $i$-connected. This is a contradiction, and then
we have $k\leq2g+i$.

\begin{proof}[Proof of Theorem \ref{theo-3}]
We first prove the case $g=0$ and $i=1$. Based on the previous
argument about the order of the zeros of eigenfunctions, the order
of vanishing on the nodal lines is less than or equal to $1$. If the
multiplicity of the first eigenvalue $\lambda^{c}_{1,\phi}(M^2)$ is
$4$, then we can find $4$ linearly independent eigenfunctions
$\psi_{i}$ satisfying
\begin{eqnarray*}
\Delta_\phi \psi_{i}
+\lambda^{c}_{1,\phi}(M^2)\psi_{i}=0,~~i=1,2,3,4.
\end{eqnarray*}
One can find $a_i$, $b_i\in \mathbb{R}$, $1\leq i\leq 3$, such that
\begin{eqnarray*}
a_{i}^{2}+b_{i}^{2}\neq 0
\end{eqnarray*}
and
\begin{eqnarray*}
(a_i\psi_{i+1}-b_{i}\psi_{1})(x_0)=0.
\end{eqnarray*}
Notice that $a_i\psi_{i+1}-b_{i}\psi_{1}$, $i=1,2,3$, are again
linearly independent. Because the tangent space is $2$-dimensional,
we can find $3$ real numbers $c_1,c_2,c_3$, which are not all zero,
such that
\begin{eqnarray*}
\sum_{i=1}^{3}c_{i}d(a_i\psi_{i+1}-b_{i}\psi_{1})(x_0)=0
\end{eqnarray*}
and
\begin{eqnarray*}
\sum_{i=1}^{3}c_{i}(a_i\psi_{i+1}-b_{i}\psi_{1})(x_0)=0.
\end{eqnarray*}
This is contradict with the fact that the order of the zero point of
the first non-trivial eigenfunction at $x_0$ must be less than or
equal to $1$. Therefore, in the case $g=0$ and $i=1$, the
multiplicity of $\lambda^{c}_{1,\phi}(M^2)$ is less than or equal to
$3$. The other cases can be proven almost the same. This is because
on $\mathbb{R}^2$ the dimension of the space of constant coefficient
partial differential operator of order less than or equal to $k$ is
equal to $\sum_{i=1}^{k+1}i$. The proof of Theorem \ref{theo-3} is
finished.
\end{proof}

\section*{Acknowledgments}
\renewcommand{\thesection}{\arabic{section}}
\renewcommand{\theequation}{\thesection.\arabic{equation}}
\setcounter{equation}{0} \setcounter{maintheorem}{0}

This research was supported in part by the NSF of China (Grant Nos.
11801496 and 11926352), the Fok Ying-Tung Education Foundation
(China), the Key Project of Jiangxi Provincial Natural Science
Foundation (Grant No. 20252BAC250003), Hubei Key Laboratory of
Applied Mathematics (Hubei University), and Key Laboratory of
Intelligent Sensing System and Security (Hubei University), Ministry
of Education.

\end{document}